%% file: colored_operads.tex
\documentclass[12pt,electronic]{amsart}                                                 
\usepackage[T1]{fontenc}
\usepackage[applemac]{inputenc}
\usepackage[english]{babel}
\usepackage[small,nohug]{diagrams}
\usepackage{url}
\usepackage[hmargin=3.7cm,vmargin=3.5cm]{geometry} 
\usepackage{graphicx}                         
\usepackage{amsmath}
\usepackage{amsthm}                      
\usepackage{amssymb}
\usepackage{amscd}     
\usepackage{mathrsfs}
\usepackage{stmaryrd}
\usepackage{enumerate}
\usepackage[small,it]{caption}
\theoremstyle{plain}                                                           
\newtheorem{thm}{Theorem}[section]

\newtheorem{prop}[thm]{Proposition}

\theoremstyle{definition}
\newtheorem{ex}[thm]{Example}
\newtheorem{rem}[thm]{Remark}
\newtheorem{defn}[thm]{Definition}
\newtheorem{notation}[thm]{Notation}

\DeclareMathOperator{\id}{id}
\DeclareMathOperator{\Hom}{Hom}

\DeclareMathOperator{\Spec}{Spec}

\DeclareMathOperator{\Aut}{Aut}

\DeclareMathOperator{\Ind}{Ind}

\newcommand{\field}[1]{\ensuremath{\mathbf{#1}}}
\newcommand{\R}{\ensuremath{\field{R}}}        
\newcommand{\Q}{\ensuremath{\field{Q}}}        
\newcommand{\C}{\ensuremath{\field{C}}}
\newcommand{\sym}{\ensuremath{\mathbb{S}}}
\newcommand{\Z}{\ensuremath{\field{Z}}}

\newcommand{\A}{\mathcal{A}}

\newcommand{\M}{\mathcal{M}}

\newcommand{\MM}{\overline{\mathcal{M}}}
\usepackage{enumerate}
\usepackage{wasysym}
\title{On the operad structure of admissible $G$-covers}
\author{Dan Petersen}\thanks{The author is supported by the G\"oran Gustafsson Foundation for Research in Natural Sciences and Medicine. }

\renewcommand{\sym}{\mathbb{S}}
\newcommand{\G}{\L G}
\renewcommand{\C}{\mathcal{C}}
\DeclareMathOperator*{\colim}{colim}

\DeclareMathOperator{\Iso}{Iso}
\DeclareMathOperator{\End}{End}

\renewcommand{\field}{\mathbb}

\newcommand{\SO}{\mathrm{SO}}

\newcommand{\V}{V}
\newcommand{\W}{W}
\newcommand{\fD}{f\mathcal{D}}
\newcommand{\graph}[1]{\mathbf{#1}}
\newcommand{\D}{\graph{di}}

\renewcommand{\SS}{\graph{S}}
\newcommand{\E}{\mathcal{E}}
\renewcommand{\M}{\mathbb{M}}
\newcommand{\ob}{\mathrm{ob}}
\renewcommand{\H}{\mathcal{H}}

\begin{document} 
  
 \maketitle

\begin{abstract} We describe the modular operad structure on the moduli spaces of pointed stable curves equipped with an admissible $G$-cover. To do this we are forced to introduce the notion of an operad colored not by a set but by the objects of a category. This construction interpolates in a sense between `framed' and `colored' versions of operads; we hope that it will be of independent interest. An algebra over this operad is the same thing as a $G$-equivariant CohFT, as defined by Jarvis, Kaufmann and Kimura. We prove that the (orbifold) Gromov--Witten invariants of global quotients $[X/G]$ give examples of $G$-CohFTs. 
\end{abstract}


\section{Introduction}
\input{introduction}
\section{Background}
\label{background}
\input{background}

\section{Operads colored by categories}\label{operads}

\input{operads}

\section{Equivariant CohFTs}\label{cohfts}
\input{cohft}



\bibliographystyle{alpha}

\bibliography{../database}


\end{document}

%% file: introduction.tex
The notion of a cohomological field theory (CohFT) was introduced by Kontsevich and Manin \cite{kontsevichmanin} as a simpler algebro-geometric relative of the notion of a $(1+1)$-dimensional topological conformal field theory, where holomorphic holes have been replaced with marked points (so one gets a theory modeled on gluing of compact Riemann surfaces along markings) and singular chains on moduli space have been replaced by homology. One can give a succinct definition of a CohFT in the language of modular operads \cite{getzlerkapranov}: a CohFT is nothing but an algebra over the modular operad $H_\bullet(\MM_{g,n},\Z)$. The main examples of CohFTs are the Gromov--Witten invariants of smooth projective varieties \cite{behrendmanin,gwinag,normalcone}. 


Jarvis, Kaufmann and Kimura \cite{jkk} defined a generalization called a $G$-CohFT, where $G$ is a finite group. Here one glues instead marked Riemann surfaces $C$ equipped with a branched covering $P \to C$ which forms a $G$-torsor away from the markings. The gluing rules need to be slightly modified: firstly because one needs a marked point on $P$ over each marked point on $C$ in order that the gluing is independent of choices, secondly because one needs to impose the condition that the monodromies around the respective markings should be inverse to each other. In algebraic language, going from CohFTs to $G$-CohFTs corresponds to going from $\MM_{g,n}$ to spaces $\MM_{g,n}^G$ of \emph{admissible $G$-covers}. One expects the main source of $G$-CohFTs to be the Gromov--Witten invariants of a global quotient $[X/G]$ (in the sense of orbifolds or stacks) of a smooth projective variety by a finite group \cite{chenruan, agv08}. Similar ideas can be found in a letter from Kontsevich to Borisov from 1996, published in \cite{abramovich08}. 


Analogous constructions have existed for a longer time in the physics literature, arising from {Chern--Simons theory with a finite gauge group}, see e.g.\ \cite{dijkgraafwitten,freed}. Also closely related is Turaev's notion of a {homotopy quantum field theory} \cite{hqft}, which is a TQFT where all spaces and cobordisms are equipped with a map up to homotopy to a fixed target space $X$. Taking $X$ a $K(G,1)$ shows the similarity with $G$-CohFTs. 

The definition of a $G$-CohFT in \cite{jkk} is unsatisfactory in one minor respect. A $G$-CohFT is defined by a list of axioms, but just as for ordinary CohFTs one would expect it to be possible to bundle together these axioms by stating that a $G$-CohFT is an algebra over a certain operad. And it is clear from the definition that a $G$-CohFT is an algebra over \emph{something}, it is just not clear in what sense the spaces $\MM_{g,n}^G$ form an operad. 

We claim that the correct definition is that $\{\MM_{g,n}^G\}$ forms a modular operad colored by a \emph{category}. The category in question is the action groupoid of $G$ acting on itself by conjugation, the so-called {loop groupoid} of the group $G$. Moreover, this groupoid carries an involution given by ``changing orientation of the loop'', which corresponds to inversion in the group, and the gluing rules need to be modified in order to accommodate this involution. 

Let us finally give a brief outline of the article. Section \ref{background} describes the moduli spaces of admissible covers and their stratifications from an operadic point of view. Section \ref{operads} contains a formal definition of a colored modular operad where the colors form a category (with duality). We have not seen this defined in the literature. Although it is quite easy to define what this should mean for an ordinary operad, it is a bit subtle to come up with the `right' definition when one considers structures defined by more general graphs than trees (that is, cyclic, wheeled, modular, etc., versions of operads). 

After this we explain in Section \ref{cohfts} how the work of Jarvis, Kaufmann and Kimura fits into this framework. We prove a result left open in their article, that the Gromov--Witten invariants of a global quotient $[X/G]$ endow the ring $H^\bullet(X,G)$ of Fantechi and G\"ottsche with the structure of a $G$-CohFT. 

In a sequel to this paper, we will extend the formalism of symmetric functions to this setting, and prove an analogue of Getzler and Kapranov's formula \cite{getzlerkapranov} for the effect of the `free modular operad' functor on the level of symmetric functions. 

%% file: background.tex

Consider first the topological version of the story: let $G$ be a (finite) group, and consider a variant of 2-dimensional TQFT modeled on sewing of compact oriented surfaces with boundary, equipped with a $G$-bundle. Then there is a basic compatibility condition needed in the definition of the sewing: for each boundary component, we get a $G$-bundle on $S^1$, and to glue surfaces we need an isomorphism between these $G$-bundles. 

In the algebraic version, there is no analogue of gluing surfaces with boundary, and one is forced to work with punctured or marked surfaces. Since the $G$-cover will not in general extend across the punctures, one is moreover forced to work with ramified covers instead. 

\begin{defn}\label{adm}Let $G$ be a finite group, and $C$ an $n$-pointed nodal curve. An \emph{admissible $G$-cover} is a covering $\pi \colon P \to C$ and a $G$-action on $P$, such that:
\begin{enumerate}
\item the quotient $P/G$ is identified with $C$ via $\pi$; 
\item the map $\pi$ is a $G$-torsor away from the nodes and markings;
\item if $x \in P$ is a node, then the stabilizer $G_x$ acts on the tangent spaces of the two branches at $x$ by characters which are inverses of each other. 
\end{enumerate}\end{defn}

Condition (3) is the algebraic analogue of the sewing condition in the topological setting. Suppose we are given two Riemann surfaces $C$ and $C'$ with marked points $y$ and $y'$. Let $\overline C$ be the nodal surface obtained by gluing $y$ and $y'$. Let $P \to C\setminus \{y\}$ and $P'\to C'\setminus \{y'\}$ be $G$-torsors. These extend uniquely to ramified covers of $C$ and $C'$, and by choosing points $x$, $x'$ in the fibers over $y$ and $y'$ they can be glued together to a covering $\overline P \to \overline C$ whenever the isotropy groups $G_x$ and $G_{x'}$ coincide. But in general the resulting covering will not be smoothable, in the sense that there is no family of $G$-covers $P_t \to C_t$ of \emph{smooth} curves, such that the limit as $t\to 0$ of this family is $\overline P \to \overline C$. Clearly, the topological obstruction to such a smoothing is that the monodromies of $P \to C\setminus \{y\}$ and  $P'\to C'\setminus \{y'\}$, computed with respect to $x$ and $x'$, are inverse to each other in $G$. This final condition is equivalent to condition (3), which however makes sense over an arbitrary base field. Nevertheless, we shall stick to the language of Riemann surfaces in this article.

Though the notion of an admissible cover predates their work (admissible covers traditionally arise when one tries to compactify moduli spaces of unramified covers, cf.\ \cite{prymschottky,harrismumford1}), Definition \ref{adm} was first written down in this form in \cite{acv03}. (They call coverings satisfying (3) \emph{balanced}. We omit this adjective, as there will be no need for unbalanced coverings.) They also construct a moduli space for such covers. This theory arises from Abramovich, Vistoli and their coauthors' work on defining Gromov-Witten invariants of stacks: it is the special case of stable maps where the target space is the stack $BG$. 

\begin{defn}We denote by $\MM_{g,n}^G$ the moduli stack parametrizing admissible $G$-covers $P \to C$ where $C$ is a stable $n$-pointed curve of genus $g$, together with a choice of a point $x_i \in P$ over every marked point $y_i \in C$.   \end{defn}

That we include liftings $x_i$ of the points $y_i$ is crucial in order for there to be a natural operad structure. 

\subsection{The operadic structure}

The spaces $\MM_{g,n}^G$ admit a kind of stratification by topological type, analogous to that of $\MM_{g,n}$. To an admissible cover $P \to C$ we associate a stable graph, namely the dual graph of $C$. The choice of a point in the fiber over each marking on $C$ produces extra structure on this graph: by considering the monodromy of the covering over each marked point, we find that the legs of the graph are decorated by elements of $G$. Condition (3) above implies that the spaces $\MM_{g,n}^G$ have partially defined analogues of the gluing maps for $\MM_{g,n}$: one can glue together two legs precisely when they have mutually inverse decorations. So it would seem that they form a kind of colored operad where there is an involution on the collection of colors. 

However, there is further structure present: the wreath product $G \wr \sym_n$ acts on $\MM_{g,n}^G$, where $\sym_n$ acts by permuting the markings and each copy of $G$ acts by changing the choice of the lifted point $x_i \in P$. Changing the point $x_i$ to $g \cdot x_i$ has the effect of changing the monodromy by conjugation with $g$. Hence $G$ acts both on the spaces involved and on the set of colors (by conjugation), and the gluing maps are equivariant for this $G$-action. 

Moreover, since there are no distinguished points in $P$ in the fibers over the nodes of $C$, we see that gluing two points together also involves simultaneously forgetting the choices of liftings over the two markings, i.e.\ quotienting by a diagonal action of $G$ acting on both markings that are being glued together. It is instructive to compare this to the framed little disks operad, which parametrizes little disks equipped with a marked point on their boundaries, and gluing involves forgetting about this marked point. 

We claim that the correct formalism for describing all this data --- the presence of a coloring, the fact that gluing means simultaneously quotienting by the action of a group acting ``on the legs'', and compatibility with the action of the group on the set of colors --- is the following: the spaces $\MM_{g,n}^G$ form a colored operad where the colors are the objects of the action groupoid $[G/G]$ in which $G$ acts on its underlying set by conjugation. Finally there is the condition of inverse monodromy, which is now most easily described as an involution of this groupoid.

\subsection{The loop groupoid}

The appearance of the groupoid $[G/G]$ is not a coincidence. For one thing, it turns out that an algebra over a $\C$-colored operad needs in particular to be a representation of $\C$. Moreover, a representation of $[G/G]$ is exactly the same as a module over the Drinfel'd (quantum) double of the group $G$. This module structure is well known in Dijkgraaf--Witten theory, cf.\ \cite{dpr, freed}. Although we will not use anything from this subsection in the rest of the paper, it seems worth giving some context.


\newcommand{\BG}{\mathbf{B}G}
\newcommand{\B}{\mathbf{B}}
\newcommand{\Cat}{\mathbf{Cat}}

\renewcommand{\L}{\mathscr{L}}

\begin{defn}Let $G$ be a group. We denote by $\L G$ the action groupoid of $G$ acting on its underlying set by conjugation, and call this the \emph{loop groupoid} of $G$. \end{defn}

\begin{rem}The groupoid $\L G$ can equivalently (and more generally) be described as the functor category $\mathrm{Fun}(\Z ,G)$, where $\Z$ and $G$ are considered as $1$-object categories. Since $|\Z| \simeq S^1$, where $|\ast|$ denotes geometric realization, this explains the terminology. \end{rem}

\begin{rem}\label{topology}One can show that for any two groupoids $\mathcal G$ and $\mathcal H$, there is a homotopy equivalence \[ |\mathrm{Fun}(\mathcal H,\mathcal G)| \simeq \mathrm{map}(|\mathcal H|,|\mathcal G|),\]
see for instance \cite{strickland}. In particular, $|\L G|$ is the space $LBG$ of free loops on the classifying space $BG$. Another way to think about this is that $\L G$ is isomorphic to the groupoid of $\mathbb C$-points of the inertia stack of $BG$, cf.\ e.g.\ \cite[Section 5]{abramovich08}. The relationship between these viewpoints is that the inertia stack $I(\mathcal X)$ is in general defined as the fiber product $\mathcal X \times_{\mathcal X \times \mathcal X} \mathcal X$. On the other hand, $LX$ is given by the homotopy pullback $X \times^h_{X\times X} X$, for any space $X$. 

In any case, this leads to a geometrically appealing situation. We are trying to combinatorially model gluing of surfaces equipped with $G$-torsors. In the topological setting, we needed for any two boundary circles an isomorphism between the respective $G$-bundles, which are (up to homotopy) points of $LBG$. Now we replace surfaces with their dual graphs, and find that we must decorate legs by $\L G$, which is a combinatorial model of $LBG$. 
 \end{rem}

The involution on $\L G$ is now defined by the obvious action of $\{\pm 1 \} = \Aut \Z$ on $\mathrm{Fun}(\Z,G)$. Topologically, inverting the generator of $\Z$ corresponds to changing orientation of the circle, in agreement with the geometric picture. 

\begin{defn}Let $\C$ be a groupoid and $k$ a field. We define the \emph{groupoid algebra} $k[\C]$ to be the $k$-algebra which is spanned as a vector space by the morphisms in $\C$, and whose product is defined on generators by 
\[ f \ast g = \begin{cases} f \circ g & \text{if this composition makes sense,} \\ 0 & \text{otherwise.} \end{cases}\]
This is extended bilinearly.\end{defn}
Just as for finite groups, $k[\C]$ is naturally a Hopf algebra. If $G$ is a finite group, then $k[\L G]$ is exactly the Drinfel'd double of the usual group algebra $k[G]$.

%% file: operads.tex
\newcommand{\op}{\mathrm{op}}

In this section we give the general definition of an operad-like structure colored by a category. By an operad-like structure we mean e.g. a cyclic or modular operad, a (wheeled) PROP, a properad, a dioperad, etc.

We begin by giving a direct definition of an ordinary operad colored by a category. For more general operad-like structures one needs some more careful combinatorics. In order to give a suitably general definition we define a category of graphs colored by some fixed category and construct the ``free operad'' functor combinatorially in terms of sums over such graphs. This functor is naturally a monad and one can then define an operad as an algebra over it. A pedagogical introduction to this point of view on operads and related structures can be found in \cite{markloperads}. 

An extra subtlety in the case of undirected graphs is that the colors should no longer just form a category, but they must be a \emph{category with duality}, i.e.\ a category $\C$ with an isomorphism $\C \cong \C^\op$ satisfying certain conditions. This is analogous to how any vector space can be an algebra over an operad, but only a vector space with an inner product can be an algebra over a cyclic operad. 

\subsection{The case of ordinary operads}

\begin{defn}\label{wreathdef}Suppose a finite group $G$ acts on a category $\C$. We define the \emph{semidirect product} $\C \rtimes G$ to be the category with the same objects as $\C$, and whose morphisms $x \to y$ are pairs $(\phi,g)$ where $g \in G$ and $\phi \in \Hom_{\C}(x,yg)$. The composition is defined by 
\[ (\phi,g) \circ (\psi,h) = ((\phi h) \circ \psi,g\cdot h).\]
\end{defn}

\begin{defn} The \emph{wreath product} $\C \wr \sym_n$ of a category with the symmetric group on $n$ letters is the semidirect product $\C^n \rtimes \sym_n$ with the obvious $\sym_n$-action. 
\end{defn}

For the remainder of this section, we fix a cocomplete symmetric monoidal category $\E$, and a small category $\C$. We shall consider operads colored by $\C$ taking values in $\E$.



\begin{defn}
A \emph{$\C\sym$-module} is a sequence $\V(n)$, $n \geq 0$, of functors  
\[ \V(n) \colon \C^\op \times (\C \wr \sym_n) \to \E. \]
\end{defn}

\begin{defn}
The \emph{tensor product} of two $\C\sym$-modules is defined by
\[ (\V\otimes \W)(n) = \coprod_{k+l=n} \Ind_{\C\wr\sym_k \times \C \wr \sym_l}^{\C\wr\sym_n} \V(k) \otimes \W(l).\]\end{defn}

By induction we mean here the left Kan extension along $\C\wr\sym_k \times \C \wr \sym_l \hookrightarrow \C\wr\sym_n$, which is the usual induction functor when $\C$ is a group. 

\begin{defn}\label{def1}
The \emph{plethysm} of two  $\C \sym$-modules is defined by the coend
\[ (\V\circ \W)(n) =  \coprod_{k\geq 0} \V(k) \otimes_{\C\wr\sym_k} \W^{\otimes k}(n)
\stackrel{\mathrm{def}}{=}  \coprod_{k\geq 0} \int^{\C\wr\sym_k} V(k) \otimes W^{\otimes k} (n) 
 \]
where $\W^{\otimes k}(n)$ is considered as a $\C^\op \wr \sym_k$-module by virtue of the fact that a $k$-fold tensor product of a representation of $\C^\op$ is a representation of $\C^\op \wr \sym_k$, using the symmetric monoidal structure on $\E$. 
\end{defn}

\begin{prop}The category of $\C\sym$-modules is monoidal with plethysm as product. \end{prop}

\begin{proof}Let $e$ be the $\C\sym$-module concentrated in degree one, where it is given by the composition
\[ \C^\op \times \C \stackrel{\mathrm{Hom}(-,-)}{\longrightarrow} \mathbf{Set} \stackrel{\phi}{\to} \E,\]
where $\phi(X) = \coprod_{x\in X} \mathbf 1$, with $\mathbf 1$ the monoidal identity in $\E$. In other words, we are forming the copower $\Hom(-,-) \odot \mathbf 1$. Then $e$ is both a left and right unit for plethysm, as one verifies using the canonical isomorphism (the ``co-Yoneda lemma'')
\[ F(x) = \int^{\mathcal C}\Hom_{\mathcal C}(-,x) \odot F(-)\]
for any functor $F$ defined on a category $\mathcal C$. Associativity is immediate from the fact that coproducts and coends can be freely commuted past each other, both being colimits. \end{proof}

\begin{ex}If $\C=G$ is a group and $\E = \text{$R$-Mod}$, then $e(1)$ is given by the group ring $R[G]$, considered as a left and right $G$-module. \end{ex}


\begin{defn}\label{defold}A \emph{$\C$-operad} is a monoid in the monoidal category of $\C \sym$-modules. \end{defn}

\begin{ex}Let $\C = X$ be a set, thought of as a discrete category. An $X$-operad is the same thing as an operad colored by the set $X$.\end{ex}

 \begin{ex}\label{fd} Let $\C = G$ be a group. A natural example here is the \emph{framed little disks operad} of \cite{getzlerbv}, for $G = \mathrm{SO}(N)$, which we claim can be thought of as a colored operad which has only one color, but where this color has a nontrivial automorphism group. 
 
Let $D_N$ be the closed unit disk in $\R^{N}$. Let $\fD_N(n)$ be the topological space parametrizing maps
\[ \coprod_{i=1}^n D_N \hookrightarrow D_N\]
where each factor is a composition of rotations, translations and positive dilations, and the images are disjoint. 
Then $\{\fD_N(n)\}$ is an $\SO(N)$-operad in $\mathbf{Spaces}$, 
with edge contractions defined by composing embeddings with each other. In particular the space $\fD_N(n)$ has an action of
\[ \SO(N)^\op \times (\SO(N)^n \rtimes \sym_n).\]
We define this action by letting the first factor act by rotating the entire disk, and the second factor act by rotations and permutations of the individual embedded disks. The gluing maps are $\SO(N)$-equivariant as required, in the sense that any gluing map is invariant under the simultaneous action of $\SO(N)$ on the input and output legs that are being glued together. 

More generally, any semidirect product operad $\mathcal{P} \rtimes G$ in the sense of \cite{salvatorewahl} is an example of a $G$-operad in our sense. The notion of a $G$-operad is, however, more general. (Note that there is an unfortunate clash of notation:  \cite{salvatorewahl} use the word $G$-operad to mean an operad in the category of spaces with a $G$-action.) \end{ex}

\begin{rem}The preceding example also demonstrates that one should really be working throughout in an enriched setting, although we have not done so for readability's sake. Indeed, we do not want to think of $\SO(N)$ as just a group, but a topological group, and we want its actions on spaces to be continuous. One should therefore consider categories enriched over some closed symmetric monoidal category $\mathcal V$ (in the preceding example, $\mathcal V = \mathbf{Spaces}$): $\E$ is a $\mathcal V$-cocomplete symmetric monoidal $\mathcal V$-category, $\C$ is a small $\mathcal V$-category, and we are given a $\mathcal V$-functor from $\C^\op \times \C \wr \sym_n$ to $\E$. All coends, copowers, Kan extensions, etc.\ need to be replaced with their $\mathcal V$-analogues. We leave the details to the reader. \end{rem}

\begin{rem}The author does not know a natural example of an operad colored by a category where that category is not in fact a groupoid. Such an example would perhaps be interesting. \end{rem}

\subsection{Categories with duality}

\renewcommand{\C}{\mathcal{C}}
\renewcommand{\D}{\mathcal{D}}

\begin{defn} A \emph{category with duality} is a category $\C$ equipped with a contravariant functor $\vee \colon \C \to \C$, 
and a natural isomorphism
\[ \eta \colon \id_\C \to \vee \circ \vee, \]
such that
the composition
\[  \vee \stackrel {\eta \vee} \longrightarrow \vee \circ \vee \circ \vee \stackrel{\vee\eta}{\longrightarrow} \vee\]
is the identity.
\end{defn}

\begin{rem}To make sense of the last equation in the preceding definition, recall that if $\epsilon \colon F \to G$ is a natural transformation, and $H$ is a contravariant functor, then the horizontal composition has reversed direction: one has $H\epsilon  \colon H G \to H F$.\end{rem}

We write $x^\vee$ rather than $\vee(x)$, where $x$ is either an object or a morphism in $\C$. 
An equivalent, more symmetric, definition is the following:

\begin{defn} A \emph{category with duality} is a category $\C$ equipped with a functor $\vee \colon \C \to \C^\op$, such that $\vee$ and $\vee^\op$ are quasi-inverses, and the  resulting counit and unit $\vee^\op \vee \to \id_\C$ and $\id_{\C^\op} \to \vee \vee^\op$ are opposites of each other. \end{defn}

\begin{ex}The category of finitely generated projective modules over a ring $A$ becomes a category with duality if we define $M^\vee = \Hom(M,A)$. More generally, any compact closed category is a category with duality. \end{ex}

\begin{ex}\label{groupex}Any groupoid is a category with duality, with $\vee$ the identity on objects and $g^\vee = g^{-1}$ on morphisms. \end{ex}

\begin{ex}A discrete category with duality is a set with an involution.\end{ex}

\begin{defn}A \emph{pairing} between two objects $x$ and $y$ of a category with duality is a morphism $\phi \colon x \to y^\vee$. (Equivalently, it is a morphism $y \to x^\vee$.) \end{defn}

\begin{defn}A pairing between $x$ and itself is said to be \emph{symmetric} if $\phi^\vee \circ \eta_x = \phi$. 
\end{defn}

\begin{ex}In the category of finitely generated projective $A$-modules, a pairing between $M$ and $N$ is a map $M \otimes N \to A$, and a symmetric pairing is a symmetric bilinear form.\end{ex}

If $\C$ and $\D$ are categories with duality, then so is the functor category $[\C,\D]$: if $F \colon \C \to \D$ is a functor, its dual is defined as $\vee_\D \circ F \circ \vee_\C$. 

\begin{defn}A \emph{weak symmetric functor} $\C \to \D$ is a functor $F$ in $[\C,\D]$ with a symmetric pairing.
\end{defn}

Explicitly, this means we have a functor $F \colon \C \to \D$ and a natural transformation 
\[ \rho\colon F \circ \vee_\C \to   \vee_\D \circ F   \]
such that the diagram
\[ \begin{diagram}
 F \circ \vee_\C  & \rTo^{\rho} & \vee_\D \circ F \\
 \dTo_{\eta_\D}  & & \uTo^{\eta_\C} \\ 
 \vee_\D \circ \vee_\D \circ F \circ \vee_\C & \rTo_\rho &   \vee_\D \circ F \circ \vee_\C \circ \vee_\C
 \end{diagram} \]
commutes. If $\rho$ is an isomorphism, then $F$ is \emph{strong symmetric}. 

\begin{ex}A weak symmetric functor from the one-object one-morphism category into $\C$ is an object of $\C$ with a symmetric pairing. \end{ex}

\begin{ex}The category $\mathbf{fdHilb}$ is naturally a category with duality, with $\vee$ the identity on objects and $T^\vee$ the adjoint of $T$. Let $G$ be a group, considered as a category with duality as in Example \ref{groupex}. A (weak or strong) symmetric functor $G \to \mathbf{fdHilb}$ is a unitary representation of $G$.  \end{ex}

\begin{ex}If $F$ is weak symmetric, then a pairing between $x$ and $y$ induces a pairing between $F(x)$ and $F(y)$.\end{ex}


\subsection{Graphs}

We shall follow the definitions and conventions of \cite{getzlerkapranov} regarding graphs, which we recall for the reader's convenience. A \emph{graph} $\Gamma$ is a finite set $F$ of \emph{flags}, a finite set $V$ of \emph{vertices}, a function $h \colon F \to V$, and an involution  $\tau$ on $F$. The fixed points of $\tau$ are called \emph{legs} and the orbits of length two are called \emph{edges}. 

A \emph{morphism of graphs} $f \colon \Gamma \to \Gamma'$ consists of two functions $f_\ast \colon V \to V'$ and $f^\ast \colon F' \to F$ such that $f^\ast$ is bijective on legs, injective on edges, and for which
\[ \begin{diagram}
F \setminus f^\ast(F') & \pile{\rTo^h \\ \rTo_{h \tau} } & V & \rTo^{f_\ast} & V'
\end{diagram}\]
is a coequalizer. Informally, $f$ is a composition of automorphisms and edge contractions. 

A graph with one vertex and no edges is called a \emph{corolla}. For every $v \in V$ we denote by $\gamma(v)$ the corolla with flag set $h^{-1}(v)$. 

A \emph{dual graph} is a graph with a \emph{genus function} $g \colon V \to \{0,1,2,\ldots\}$. We denote by $n(\Gamma)$ the number of legs of a graph $\Gamma$. For a vertex $v$, we  use the shorthand $n(v) = n(\gamma(v))$. A \emph{morphism of dual graphs} is a morphism  $f \colon \Gamma \to \Gamma'$ of the underlying graphs such that for all $v' \in V'$ we have
\[ 2g(v')-2 + n(v') = \sum_{f_\ast(v) = v'} (2g(v)-2 + n(v)). \]  If $\Gamma$ is a dual graph, then we declare its \emph{genus} $g(\Gamma)$ to be the unique integer satisfying
\[2 g(\Gamma)  - 2 + n(\Gamma) = \sum_{v\in V} (2g(v)-2+n(v)).\] A simple lemma shows that if $f\colon \Gamma \to \Gamma'$ is a morphism of dual graphs, then $g(\Gamma) = g(\Gamma')$. A dual graph is called \emph{stable} if for each vertex $v$ the inequality 
\[ 2g(v) - 2 + n(v) > 0 \]
is satisfied.

\begin{rem}The idea of a dual graph is best thought of topologically as follows. We imagine that a vertex of genus $g$ with $n$ adjacent legs describes a compact oriented surface of genus $g$ with $n$ boundary circles. Then the number $2g-2+n$ is just the negative of the Euler characteristic of the surface. If we think of an edge contraction as an operation which glues together the corresponding boundary components, then the formulas in the definition of a dual graph express that Euler characteristic should be additive over gluing of circles. \end{rem}

Now fix a category with duality $\C$.

\begin{defn}A \emph{$\C$-graph} is a graph $\Gamma$ with the following extra data: for every flag $x$ we are given an object $A_x$ of $\C$, and for an edge connecting the flags $x$ and $y$ we are given a pairing between $A_x$ and $A_y$. \end{defn}

\begin{defn}\label{morph} A \emph{morphism of $\C$-graphs} is a morphism $\Gamma \to \Gamma'$ of underlying graphs, together with a morphism $q_x \colon A_{f^\ast(x)} \to A_x $ for every flag $x$ of $\Gamma'$, such that for an edge between $x$ and $y$ in $\Gamma'$, the following diagram commutes:
\[\begin{diagram}
A_{f^\ast(x)} & \rTo & A_{f^\ast(y)}^\vee \\
\dTo^{q_x} & & \uTo_{q_y^\vee} \\
A_x & \rTo & A_y^\vee   
\end{diagram}
\]
\end{defn}

\begin{rem}One can describe a $\C$-graph as a graph $\Gamma$ together with a symmetric functor $\mathcal F \to \C$, where $\mathcal F$ is an appropriate category with duality defined in terms of the flags and edges of $\Gamma$. Then a morphism of $\C$-graphs can be defined more simply in terms of a natural transformation. We leave the details to the reader. \end{rem}

\subsection{Operads as algebras}


\newcommand{\preop}{{[\SS^0,\E]}}
  
\begin{notation} Let $\SS$ be the category of stable $\C$-graphs. Let $\SS^0$ be the full subcategory of corollas in $\SS$. Let $\preop$ denote the category of functors $\SS^0 \to \E$.  \end{notation}

\begin{defn}We call the objects of $\preop$ \emph{stable $\C\sym$-modules}. \end{defn}

\begin{rem}Suppose $\C$ is trivial. Then a functor $\SS^0 \to \E$ is the same thing as a \emph{stable $\sym$-module} in the terminology of \cite{getzlerkapranov}, as $\SS^0$ has the obvious skeleton
\[\SS^0 \cong \coprod_{\substack{g,n \geq 0\\2g-2+n>0}} \sym_n.  \]
Hence a functor from $\SS^0$ to $\E$ is just a family of $\sym_n$-representations indexed by $g$ and $n$, which recovers the definition of Getzler and Kapranov and justifies our terminology. More generally one has for any $\C$ that 
 \[\SS^0 \cong \coprod_{\substack{g,n \geq 0\\2g-2+n>0}} \C \wr \sym_n. \]\end{rem}

\renewcommand{\Iso}{\mathrm{Bij}}

\begin{notation}Let $\Iso(\SS)$ denote the full subcategory of $\SS$ consisting of graph morphisms which do not contract any edge. \end{notation}

\begin{rem}\label{isorem}Any functor $\V \colon \SS^0 \to \E$ can be extended to a functor $\Iso(\SS) \to \E$ via
\[ \V(\Gamma) = \bigotimes_{v \in V(\Gamma)} \V(\gamma(v)).\]Note that if $\Gamma$ is stable then so are all the $\gamma(v)$. \end{rem}

\begin{defn}\label{def2}Let $\M$ be the endofunctor on $\preop$ defined by 
\[ \M \V(\gamma) = \colim_{\Gamma \in \Iso(\SS) \downarrow \gamma} \V(\Gamma)\]
for any corolla $\gamma$. Here $\Iso(\SS) \downarrow \gamma$ denotes the slice category over $\gamma$; its objects are graphs in $\SS$ with a map to $\gamma$, and its morphisms are morphisms over $\gamma$ which do not contract any edges. \end{defn}

For any corolla $\gamma \in \SS^0$ there is a natural map $\V(\gamma) \to \M \V(\gamma)$ induced by sending $\id_\gamma$ to the corresponding morphism in $\Iso(\SS) \downarrow \gamma$. This defines a natural transformation $\eta \colon \id_\preop \to \M$. There is also a natural transformation $\mu \colon \M^2 \to \M$, defined as usual by `erasing braces' (cf.  \cite{markloperads}).

\begin{prop}The functor $\M$ is a monad with unit $\eta$ and multiplication $\mu$.  \end{prop}

\begin{proof}A rather conceptual proof can be found in \cite{getzlerkapranov}, which carries through with only minor changes to the $\C$-colored setting. The necessary commutative diagrams can also be checked somewhat tediously by hand. \end{proof}

\begin{defn}A \emph{modular $\C$-operad} is an $\M$-algebra.\end{defn}

\begin{rem}A posteriori, the fact that $\M$ turns out to be a monad can be explained by saying that $\M$ maps a stable $\C\sym$-module $\V$ to the underlying stable $\C\sym$-module of the free modular $\C$-operad generated by $\V$. Hence the fact that $\M$ is a monad expresses the fact that the free modular operad functor is left adjoint to the forgetful functor sending a modular operad to its underlying stable $\C\sym$-module.
 \end{rem}

\begin{rem}\label{altdef}One can describe modular $\C$-operads more explicitly in the following way. A modular $\C$-operad $\A$ consists of:
\begin{enumerate}
\item for any $g, n \geq 0$ such that $2g-2+n > 0$, and any $n$-tuple $(x_1,\ldots,x_n)$ of objects of $\C$, an object 
\[ \A(g,x_1,\ldots,x_n)\]
of $\E$;
\item for any $\sigma \in \sym_n$ a map
\[ \A(g,x_1,\ldots,x_n) \to \A(g,x_{\sigma(1)},\ldots,x_{\sigma(n)});\]
\item for any morphism $x_i \mapsto x_i'$ in $\C$ a map
\[ \A(g,x_1,\ldots,x_i,\ldots,x_n) \to \A(g,x_1,\ldots,x_i',\ldots,x_n); \]
\item for any $i$ and $j$ and for every pairing between $x_i$ and $y_j$, a gluing map
\[ \A(g_1,x_1,\ldots,x_n) \otimes \A(h,y_1,\ldots,y_m) \to \A(g+h,x_1,\ldots,\widehat{x_i},\ldots,\widehat{y_j},\ldots,y_m);\] 
\item for any $i \neq j$ and for every pairing between $x_i$ and $x_j$, a gluing map
\[ \A(g,x_1,\ldots,x_n) \to \A(g+1,x_1,\ldots,\widehat{x_i},\ldots,\widehat{x_j},\ldots, x_n).\]
\end{enumerate}
One thinks of $\A(g,x_1,\ldots,x_n)$ as the value of $\A$ on a corolla of genus $g$ with $n$ legs decorated by $x_1,\ldots,x_n$. We will not list the functoriality conditions and commutative diagrams that these maps must satisfy. 
\end{rem}

\subsection{Algebras over operads}

The notion of an algebra over an operad can be defined in various levels of generality. We assume in this section that the target category $\E$ is compact closed, i.e.\ every object is dualizable, which will be sufficient for this article. In particular, this implies that $\E$ is a category with duality. 

\begin{defn}Suppose given a weak symmetric functor $\rho \colon \C \to \E$. We associate to $\rho$ its \emph{endomorphism operad} $\End_\rho$. In the notation of Remark \ref{altdef}, it is defined on objects by
\[ \End_\rho(g,x_1,\ldots,x_n) = \bigotimes_{i=1}^n \rho(x_i).\]
Every pairing between $x$ and $y$ in $\C$ gives a pairing between $\rho(x)$ and $\rho(y)$ in $\E$ in the usual sense, i.e.\ a map
\[ \rho(x) \otimes \rho(y) \to \mathbf 1\]
where $\mathbf 1$ is the monoidal unit in $\E$. This pairing defines the gluing maps for the modular $\C$-operad $\End_\rho$. \end{defn}

\begin{defn}An \emph{algebra} over a modular $\C$-operad $\A$ is a weak symmetric functor $\rho \colon \C\to\E$ and a morphism $\A \to \End_\rho$. \end{defn}

\subsection{Other operad-like structures}
\renewcommand{\G}{\mathbf{G}}
By considering some other category of graphs $\G$ instead of $\SS$ one can define in a similar same way $\C$-colored versions of other operad-like constructions. One lets $\G^0$ be the subcategory of corollas. In order for the definition of $\M$ to make sense, one needs to assume that for any $\Gamma \in \ob(\G)$ and $v \in V(\Gamma)$, we also have $\gamma(v) \in \ob(\G)$. To define the multiplication map $\mu$ one needs to assume that $\G$ is closed under ``erasing braces''. With these assumptions, it will remain true that $\M$ is a monad.

For example, take $\G$ to be the full subcategory of trees in $\SS$. The algebras over the corresponding monad are exactly the \emph{cyclic $\C$-operads}. 

We would also like to be able to define $\C$-colored versions of more ordinary things like operads and PROPs, which are modeled on directed graphs. One could repeat appropriate modification of all our definitions for digraphs, but there is a quicker way. This is based on the observation that an ordinary operad is the same thing as a two-colored cyclic operad whose colors are $\{\mathrm{input},\mathrm{output}\}$, and where the gluing rules have been twisted by an involution: one is only allowed to glue an input leg to an output, and vice versa. 

Observe that for any category $\C$, there is an obvious structure of category with duality on the disjoint union $ \C \coprod \C^\op.$

\begin{defn}We define a \emph{$\C$-digraph} to be a $(\C\coprod\C^\op)$-graph. Flags decorated by objects in $\C$ are called \emph{incoming} and flags decorated by objects in $\C^\op$ are \emph{outgoing}. \end{defn}

\begin{rem}Note that every edge in a $\C$-digraph consists of exactly one incoming and one outgoing flag, by our definition of a pairing. \end{rem}

Let then for instance $\G$ be the category of $\C$-digraphs which are trees, and where each vertex is adjacent exactly one outgoing flag. Algebras over the resulting monad are called $\C$-operads. If $\G$ consists of arbitrary $\C$-digraphs which are trees, then we have defined the notion of a $\C$-PROP. This also gives the correct notions of algebras over $\C$-operads and $\C$-PROPs. 

\begin{prop}This definition of a $\C$-operad coincides with the one in Definition \ref{defold}. \end{prop}

\begin{proof}We allow ourselves to be brief, as the proof is similar to the uncolored case \cite[Theorem 40]{markloperads}. The only new subtlety in the $\C$-colored situation is that we must compare the coend appearing in Definition \ref{def1} with the colimit in Definition \ref{def2}. 

Consider the full subcategory $\mathcal G$ of $\Iso(\G) \downarrow \gamma$ where the underlying graph is given by some fixed graph $\Gamma$ with a single edge. An object of $\mathcal G$ consists of a decoration of this edge, i.e.\ two objects $x$ and $y$ of $\C \coprod \C^\op$, and a pairing between $x$ and $y$. It follows that an object of $\mathcal G$ is an arrow in $\C$. By comparing with Definition \ref{morph}, we see that a morphism between $x\to y$ and $x' \to y'$ is a commutative square
\begin{diagram}
x & \rTo & y \\
\dTo & & \uTo \\
x' & \rTo & y'.
\end{diagram}
In other words, $\mathcal G$ coincides with the so-called \emph{twisted arrow category} of $\C$, with its natural map to $\C^\op \times \C$. If $F$ is any functor on $\C^\op \times \C$, then
\[ \colim_{\mathcal G} F = \int^\C F,\]
see \cite[Ex. IX.6.3]{maclane}. For a graph $\Gamma$ with $n$ edges, we find instead the category $\C \wr \sym_n$, and the coend over $\C \wr \sym_n$. It is now not hard to show that the two definitions of a $\C$-operad coincide.
\end{proof}

\renewcommand{\G}{\L G}

%% file: cohft.tex
\subsection{The definition of a $G$-CohFT}
\renewcommand{\M}{\mathcal{M}}
\renewcommand{\S}{\graph{S}}

\newcommand{\Mot}{\mathbf{Mot}}
\newcommand{\stack}[1]{\mathcal{#1}}
\newcommand{\X}{\stack{X}}
\newcommand{\Y}{\stack{Y}}
\newcommand{\ev}{\mathrm{ev}}
\newcommand{\vir}{\mathrm{vir}}
\newcommand{\gl}{\mathrm{gl}}
\newcommand{\stacks}{\mathbf{Stack}}

Recall that $\MM_{g,n}^G$ is the moduli stack parametrizing stable $n$-pointed curves $C$ of genus $g$ equipped with an admissible $G$-torsor $P \to C$ and liftings of the $n$ markings to $P$. Let $\S$ be the category of stable $\L G$-graphs, and again $\S^0$ the full subcategory of corollas. Let $\stacks$ be the category of DM-stacks over some fixed base $k$ where $|G|$ is invertible. The analytically inclined reader can also take $\stacks$ to be the category of complex orbispaces.  

\begin{rem}There are two minor issues at this point. We wish to consider operads in $\stacks$.  Unfortunately, we formulated the earlier theory in a cocomplete symmetric monoidal category, but $\stacks$ is not cocomplete, and it is a $2$-category! However, neither of these are serious problems. First of all, even though $\stacks$ is not cocomplete, all colimits that occur in the definition of an $\L G$-modular operad will exist: indeed, whenever the category of colors is a finite groupoid, it is easy to see that one only needs to assume the existence of coproducts and quotients by actions of finite groups. Secondly, there are no $2$-categorical surprises, either: if we let $\E$ be a 2-category instead throughout the preceding section, then the endofunctor $\mathbb M$ is naturally a $2$-monad, and we can define a modular operad to be a pseudo-algebra over it.  \end{rem}

\begin{defn}For a corolla $\gamma \in \ob(\S^0)$ with genus $g$, and legs decorated by $\gamma_1,\ldots,\gamma_n$, let $\M(\gamma)$ be the open and closed substack of $\MM_{g,n}^G$ where the monodromy around the $i$th marking is given by $\gamma_i$, for $i = 1,\ldots,n$. Then $\M$ naturally becomes a stable $\L G \wr \sym$-module in $\stacks$. \end{defn}

\begin{thm}The functor $\M$ extends naturally to a modular $\L G$-operad in $\stacks$. \end{thm}

\begin{proof}The structure maps in the operad $\M$ are given by gluing together admissible covers along markings. The monodromy condition ensures that this is well defined. For the necessary associativity conditions, apply the 2-Yoneda lemma: on the level of moduli functors, associativity is clear. \end{proof}

Since homology is a symmetric monoidal functor, one immediately obtains a modular $\L G$-operad $H_\bullet(\M)$ in the category of graded $\field Z$-modules (assuming that we are working over the complex numbers). Algebraically, it is maybe more natural to consider the co-operad $H^\bullet(\M)$ associated to some Weil cohomology theory. In any case one can consider (co)algebras over the resulting operads. The main examples of such algebras are the $G$-equivariant Cohomological Field Theories of \cite{jkk}. They assume the existence of a flat identity, which is not always natural from the operadic perspective. If we agree that a non-unital CohFT is defined by omitting axioms (iii) and (iv) from Definition 4.1 in loc.\ cit., then we can state the following result.

\begin{prop}An algebra $\H$ over $H_\bullet(\M,\Q)$ (in the category of finite dimensional vector spaces) is the same thing as a non-unital $G$-CohFT.\end{prop}

\begin{proof}The usual proof that an algebra over $H_\bullet(\MM_{g,n})$ is the same thing as a CohFT carries through with only minor changes. \end{proof}

In particular $\H$ needs to be a representation of $\L G$, which means that it is a module over the Drinfel'd double $D(k[G])$. As remarked earlier, this module structure is well known by physicists \cite{dpr, freed, discretetorsion}. 

\begin{rem}Axiom (i), that $\H$ is a $G$-graded $G$-module, just says that $\H$ is a representation of $\L G$. Write $\H = \bigoplus_{\gamma \in G} \H_\gamma$. We remark that any algebra over $H_\bullet(\M,\Q)$ has a natural structure of a non-unital braided commutative $G$-Frobenius algebra obtained by imitating the construction in \cite{jkk}. The multiplication is defined by noting that $\MM_{0,3}^G$ is a finite union of points (generally with nontrivial automorphism group), each of which defines a partial multiplication on $\H$:
\[ \H_{\gamma_1} \otimes \H_{\gamma_2} \to \H_{\gamma_3},\]
where $\gamma_i$ is the monodromy around the $i$th marked point. A total multiplication can then be defined by summing over the distinguished points $\xi(\gamma_1,\gamma_2, \gamma_2^{-1}\gamma_1^{-1})$, see \cite[2.5]{jkk}. The arguments of loc.\ cit.\ extend to show associativity (i.e.\ the WDVV equation, via $\MM_{0,4}^G$) and the trace axiom (via $\MM_{1,1}^G$). \end{rem}

\subsection{Gromov--Witten invariants of global quotients}Just as the main example of a CohFT is the cohomology of a smooth projective variety, it is expected that the main example of a $G$-CohFT comes from a smooth projective variety with a $G$-action. So let for the remainder of this section $X$ be a smooth projective variety acted upon by $G$. For simplicity, we work over the complex numbers, so that classes of curves lie in the second integral homology group; it is well known how to describe this algebraically. 

\begin{defn}Let $\beta \in H_2(X/G,\Z)$. Define $\MM_{g,n}^G(X,\beta)$ to be the moduli stack parametrizing the following data: 
\begin{itemize}
\item an admissible $G$-cover $P \to C$, where $C$ is a \emph{prestable} $n$-pointed curve of genus $g$ 
\item a $G$-equivariant map $f \colon P \to X$, such that the induced map $\overline f \colon C \to X/G$ is stable in the sense of Kontsevich and $\overline f_\ast [C] = \beta$; 
\item a section of $P\to C$ over each marked point of $C$. 
\end{itemize}
Equivalently, we have $\MM_{g,n}^G (X,\beta) = \MM_{g,n}([X/G],\beta) \times_{\MM_{g,n}(BG)} \MM_{g,n}^G$, where $\MM_{g,n}(\mathcal X,\beta)$ denotes the usual space of stable maps to a stack. \end{defn}


It follows from the work of \cite{normalcone} and \cite{agv08} that $\MM_{g,n}^G(X,\beta)$ has a virtual fundamental class $[\MM_{g,n}^G(X,\beta)]^\vir$ defined by the relative obstruction theory given by the $G$-invariants of $\mathbf R \pi_\ast f^\ast T_X$, where $\pi \colon P \to \MM_{g,n}^G(X,\beta)$ is the natural projection.

\begin{defn}Denote by $\M(X,\beta)$ the stable $\L G \wr \sym$-module in $\stacks$ given by the spaces $\MM_{g,n}^G(X,\beta)$. We extend $\M(X,\beta)$ to a functor from stable $\L G$-graphs to stacks, but in a slightly different way than in Remark \ref{isorem}: for an $\L G$-graph $\Gamma$ with $n$ vertices, we define 
\[ \M(X,\beta)(\Gamma) = \coprod_{\beta_1 + \ldots + \beta_n = \beta} \prod_{v \in V(\Gamma)} \M(X,\beta_i)(\gamma(v)). \]
\end{defn}

\newcommand{\I}{\mathrm{I}}

\begin{defn}
The \emph{inertia variety} of $X$ is defined by 
\[ \I X = \coprod_{g\in G} X^g.\]\end{defn}

Note that $\I X$ is naturally a representation of $\L G$ in the category of algebraic varieties, since the element $h \in G$ carries $X^g$ to $X^{hgh^{-1}}$. 

Since $X$ is smooth, its inertia variety is smooth too, see \cite{iversenfixed}.

\newcommand{\Corr}{\mathbf{Corr}}

\begin{defn}Let $\Corr$ be the $\Q$-linear category, whose objects are smooth and proper DM-stacks, and whose morphisms are given by
\[ \Hom_\Corr(\X,\Y) = A^{\bullet}(\Y\times \X), \]
where the latter denotes the Chow ring with rational coefficients. Composition is defined via the formula
\[ f \circ g = p_{13,\ast} (p_{12}^\ast f \cup p_{23}^\ast g). \]\end{defn}
\begin{rem}The category of spans of smooth proper DM-stacks, with morphisms defined via pullbacks, sits naturally inside $\Corr$: a span
\[ \X \stackrel{f}{\leftarrow} \stack Z \stackrel g \to \Y \]
defines a morphism $\X \to \Y$ in $\Corr$ via $(g \times f)_\ast [\stack Z]$. 
\end{rem}

\begin{rem}Let $\Corr'$ be the category defined in the same way, except with varieties instead of stacks. The natural inclusion $\Corr' \hookrightarrow \Corr$ induces an equivalence of categories once one takes the pseudo-abelian completion of both categories, see \cite{toenmotives}. \end{rem}

The category $\Corr$ is compact closed with every object equal to its own dual. The counit is given by the span
\[ \X \times \X \stackrel \Delta \leftarrow \X \to \Spec k,\]
and vice versa for the unit. This is a kind of motivic Poincar\'e duality; it gives the usual Poincar\'e duality on any realization functor $H^\bullet$. Moreover,  $\I X$ is a symmetric functor $\L G \to \Corr$ since $X^g = X^{g^{-1}}$. It follows that we can talk about the endomorphism operad $\End(\I X)$, which is a modular $\L G$-operad in $\Corr$. Its value on an $n$-tuple $(g_1,\ldots,g_n)$ of elements of $G$ is the product $\prod_{i=1}^n X^{g_i}$.

There are natural evaluation maps $\MM_{g,n}^G(X,\beta) \to \I X$, giving a diagram
\begin{equation*} \MM_{g,n}^G \leftarrow \MM_{g,n}^G(X,\beta) \to (\I X)^n, \end{equation*}
equivariant for the $\L G \wr \sym_n$-action on all three spaces. We can write this as a diagram of stable $\L G \wr \sym$-modules in $\stacks$:
\[ \M \stackrel \pi \leftarrow \M(X,\beta) \stackrel \ev \to \End(\I X). \]
Pushing forward the virtual fundamental class defines a morphism $\M \to \End(\I X)$ of $\L G \wr \sym$-modules in $\Corr$,
\[ (\ev \times \pi)_\ast [\M(X,\beta)]^\vir \in A^\bullet(\End(\I X) \times \M) .\]

\begin{thm}For any fixed $\beta \in H_2(X/G,\Z)$, the morphism just defined gives the inertia variety $\I X$ the structure of an algebra over $\M$ in $\Corr$. \end{thm}

\begin{proof}We need to show that for any morphism $\Gamma \to \Gamma'$ in $\S$, the diagram  
\[ \begin{diagram} 
\M(\Gamma') & \rTo & \End(\I X)(\Gamma') \\
\uTo & & \uTo \\ 
\M(\Gamma) & \rTo & \End(\I X) (\Gamma) 
\end{diagram}\]
in $\Corr$ commutes. We may assume that $\Gamma \to \Gamma'$ is given by a contracting a single edge, which is decorated by $g, g^{-1} \in G$. In this case we have 
\[ \End(\I X)(\Gamma) = \End(\I X)(\Gamma') \times X^g \times X^{g^{-1}}.\]
Unwinding the definition of composition in $\Corr$, we see that we must study the following diagram in $\stacks$:
\newarrow{Id}{=}{=}{=}{=}{=}

\begin{diagram}[labelstyle=\scriptstyle]
A & \rTo & \M(X,\beta)(\Gamma') & \rTo &  \End(\I X)(\Gamma') \\
\dTo & \square & \dTo & &  \\
\M(\Gamma) & \rTo_{\rm gl} & \M(\Gamma') & &  \uTo \\
 & & & &  \\
\dId & & \End(\I X ) (\Gamma)  & \lTo^{\id \times \Delta} & \End(\I X)(\Gamma') \times X^g \\ 
 & & \uTo & \square & \uTo \\ 
 \M(\Gamma) & \lTo & \M(X,\beta)(\Gamma) & \lTo & B.
\end{diagram}
Here $\Delta$ is the diagonal map $X^g \to X^g \times X^{g^{-1}} = X^g \times X^g$, and $\gl$ is the gluing map of the operad $\M$ in $\stacks$. The spaces $A$ and $B$ are defined by the requirement that the smaller squares are cartesian. What we need to show is that the pushforwards of $\gl^! [\M(X,\beta)(\Gamma')]^\vir$ and $\Delta^! [\M(X,\beta)(\Gamma)]^\vir$ to $A^\bullet(\End(\I X)(\Gamma') \times \M(\Gamma) )$ coincide. 

There is a natural morphism $h \colon B \to A$, which is not an isomorphism. Indeed, after unwinding the fiber products one finds that $B$ parametrizes all the same data as $\M(X,\beta)(\Gamma')$, together with a decomposition of the admissible cover $P \to C$ into two components whose genera and markings are determined by $\Gamma$. The stack $A$ parametrizes the same thing, except one only has a decomposition of the \emph{stabilization} of $P \to C$ into two components. However, one can show that $h$ is an isomorphism on an open set, and then prove that $h_\ast \Delta^! [\M(X,\beta)(\Gamma)]^\vir = \gl^! [\M(X,\beta)(\Gamma')]^\vir$, which proves the claim. What we need are exactly the properties (III) and (IV) in \cite{behrendmanin}, which they refer to as `cutting edges' and `isogenies'. These are not proven exactly in this form in \cite{agv08}, but they follow by combining \cite[Proposition 5.3.1, 5.3.2]{agv08} and the arguments of \cite[Proposition 8]{gwinag} and the calculation immediately following Lemma 10 in loc.\ cit., which generalize from pre-stable pointed curves to pre-stable pointed curves with an admissible cover. \end{proof}

\begin{defn}We define $\Theta_X$ to be the usual Novikov ring of $X/G$, i.e. the ring of formal power series in the variables $q^\beta$, where $\beta \in H_2(X/G,\Z)$ is the class of a curve, and $q^{\beta}q^{\beta'} = q^{\beta+\beta'}$. \end{defn}

\begin{defn}Let $\Corr \otimes \Theta_X$ be the category obtained by tensoring all hom-spaces in $\Corr$ with $\Theta_X$.  \end{defn}

We define a morphism $\phi \colon \M \to \End (\I X)$ in $\Corr \otimes \Theta_X$ by 
\[ \sum_{\beta }(\ev \times \pi)_\ast [\M(X,\beta)]^\vir q^\beta \in A^\bullet(\End(\I X) \times \M)\otimes \Theta_X. \]

\begin{thm}With these maps, $\I X$ is an algebra over $\M$ in $\Corr \otimes \Theta_X$. \end{thm}

\begin{proof}This is clear from the preceding theorem.\end{proof}

The category $\Corr$ is equipped with realization functors associated to (Weil) cohomology theories; similarly, the category $\Corr \otimes \Theta_X$ has functors $Y \mapsto H^\bullet(Y,\Theta_X)$ by the universal coefficients theorem. The cohomology of $\I X$ is exactly Fantechi and G\"ottsche's ring $H^\bullet(X,G)$. Applying $H^\bullet$ to the morphism $\M \to \End \I X$, one finds the following result:

\begin{thm}Let $X$ be a smooth projective variety with an action of the finite group $G$. Then the stringy cohomology ring $H^\bullet(X,G)$, taken with coefficients in the Novikov ring of $X$, is in a canonical way a $G$-CohFT. \end{thm}

\begin{rem}In the above statement, we consider $H^\bullet(X,G)$ just as a super vector space, but one can with some care introduce a grading compatible with the algebra. To do this, one needs to introduce a grading on $\Theta$ via $\deg(q^\beta) = -2c_1[X/G]\cap \beta$, and equip $H^\bullet(X,G)$ with the so-called age grading. We omit the details as this is well known. \end{rem}

The above theorem was announced in \cite{jkk}, but a proof has not appeared. Although it is certainly possible to prove this without the language of operads, the author believes that the operadic framework has simplified the proof.

\renewcommand{\M}{\mathbb{M}}